\UseRawInputEncoding
\documentclass[a4paper, 12pt,reqno]{amsart}
\usepackage{amsmath,amssymb,amsthm,enumerate}%,backref,cite
\allowdisplaybreaks
\theoremstyle{plain}
\newtheorem{theorem}{Theorem}
\newtheorem*{theorem*}{Theorem}
\newtheorem{lemma}{Lemma}
\newtheorem*{lemma*}{Lemma}
\theoremstyle{definition}

\newtheorem*{definition*}{Definition}
\theoremstyle{remark}
\newtheorem{remark}{Remark}

\newtheorem*{remark*}{Remark}
\newtheorem{conjecture}{Conjecture}
\newtheorem*{statement*}{Statement}

\begin{document}
\title[Modificated Salem functions]{On one class of modifications of the Salem function}
%{One modification of the Salem function}
%{A certain modification of classical singular function}

\author{Symon Serbenyuk}



\maketitle
\text{\emph{simon6@ukr.net}}\\
\text{\emph{Kharkiv National University of Internal Affairs, Ukraine }}
\begin{abstract}

The present article deals with properties of one class of functions with complicated local structure. These functions can be modeled by certain operators  of digits. Such operators were considered by the author earlier (for example, see \cite{Serbenyuk-2016, Symon2023} and references therein). This research is a generalization of investigations presented in the last-mentioned papers.
\end{abstract}

\section{Introduction}

One can begin with quoting  Henri Poincar\'e:

``Logic sometimes breeds monsters. For half a century there has been springing up a host of weird functions, which seem to strive to have as little resemblance as possible to honest functions that are of some use. No more continuity, or else continuity but no derivatives, etc. More than this, from the point of view of logic, it is these strange functions that are the most general; those that are met without being looked for no longer appear as more than a particular case, and they have only quite a little corner left them.

Formerly, when a new function was invented, it was in view of some practical end. To-day they are invented on purpose to show our ancestors' reasonings at fault, and we shall never get anything more than that out of them.

If logic were the teacher's only guide, he would have to begin with the most general, that is to say, with the most weird, functions. He would have to set the beginner to wrestle with this collection of monstrosities. If you don't do so, the logicians might say, you will only reach exactness by stages. ---  Henri Poincaré, Science and Method (1899), (1914 translation), page 125" (\cite{Wikipedia-pathology}).

``In mathematics, when a mathematical phenomenon runs counter to some intuition, then the phenomenon is sometimes called pathological" (\cite{Wikipedia-pathology}).  Really, in analysis some examples of  pathological objects exist. There are fractal sets (see \cite{Bunde1994, Falconer1997, Falconer2004, Mandelbrot1977, Mandelbrot1999, Moran1946, sets1, sets2, sets} and references therein), as well as functions with complicated local structure such as singular (for example, \cite{{Salem1943}, {Zamfirescu1981}, {Minkowski}, {S.Serbenyuk 2017}}),  nowhere monotonic \cite{Symon2017, Symon2019}, and nowhere differentiable functions  (for example, see \cite{{Bush1952}, {Serbenyuk-2016}}, etc.).  An interest in such functions can be explained by their  connection with modelling  real objects, processes, and phenomena (in physics, economics, technology, etc.) and with different areas of mathematics (for example, see~\cite{BK2000, ACFS2011, Kruppel2009, OSS1995,    Symon21, Symon21-1, Sumi2009, Takayasu1984, TAS1993, Symon2021}). ``Since Poincar\'e, nowhere differentiable functions have been shown to appear in basic physical and biological processes such as Brownian motion and in applications such as the Black-Scholes model in finance" (\cite{Wikipedia-pathology}).

As noted in \cite{Symon2023}, researchers are trying to find simpler examples of  functions with complicated local structure. For example (see \cite{S. Serbenyuk systemy rivnyan 2-2} and references therein), in 1830, the first example of a continuous  non-differentiable  function was modeled by Bolzano in  ``Doctrine on Function" but the last paper was published one   hundred years later.  Brief historical remarks on  functions with complicated local structure are given in~\cite{{ACFS2017}, {S. Serbenyuk systemy rivnyan 2-2}}.  In the paper \cite{Salem1943}, Salem introduced the following one of the simplest examples of singular functions:
$$
S(x)=S\left(\Delta^q _{i_1i_2...i_k...}\right)=\beta_{i_1}+ \sum^{\infty} _{k=2} {\left(\beta_{i_k}\prod^{k-1} _{r=1}{p_{i_r}}\right)}=y=\Delta^{P_q} _{i_1i_2...i_k...},
$$
where $q>1$ is a fixed positive integer, $p_j>0$ for all $j=\overline{0,q-1}$, and $p_0+p_1\dots + p_{q-1}=1$. Here
$$
\Delta^q _{i_1i_2...i_k...}:=\sum^{\infty} _{k=1}{\frac{i_k}{q^k}}, ~~~i_k\in\{0, 1, 2, 3, \dots , q-1 \}.
$$
Let us note that an arbitrary value (any number from the closed interval $[0,1]$) of the Salem function can be represented by the number notation $\Delta^{P_q} _{i_1i_2...i_k...}$ for a fixed positive integer $q>1$. The last representation is called \emph{the $P_q$-representation of $x$}. In this paper, the main attention is given to functions whose arguments and values represented by this representation.

Finally, one can remark that generalizations of the Salem function can be singular, non-differentiable functions,  or those that do not have a derivative on a certain set. There exist  a number of investigations which are devoted to the Salem function and its generalizations or modifications in terms of various representations of an argument (for example, see \cite{ACFS2017, Kawamura2010, Symon2015, Symon2017, Symon2019, Symon2021, Symon2023} and references in these papers). 

This paper is devoted to some modifications of the Salem function. The present technique of modifications was introduced in 2012 in  \cite{{S. Serbenyuk abstract 6}, Symon12(2)} in terms of the ternary representation, was considered in \cite{{S. Serbenyuk abstract 7}, {S. Serbenyuk abstract 8}, {Serbenyuk-2016}, {S. Serbenyuk systemy rivnyan 2-2}} for constructing functions  in terms of the $q$-ary and nega-$q$-ary representation, as well as in~\cite{Symon2023} was  investigated for modelling functions with complicated local structure in terms of representations of numbers by the Salem function with $q=3$.

The present research is a  generalization of~\cite{Symon2023}, \cite{Serbenyuk-2016} and  includes the consideration of  self-similar, fractal, integral, and differential, etc. properties of modifications of  the Salem function for an arbitrary positive integer $q>1$.

%%%%%%%%%%%%%%%%%%%%%%%%%%%%%%%%%%%%%%%%%%%%%%%%%
\section{The main object}
%%%%%%%%%%%%%%%%%%%%%%%%%%%%%%%%%%%%%%%%%%%%%%%%%

Suppose $q>1$ is a fixed positive integer,  as well as $P_q :=\{p_0, p_1, \dots, p_{q-1}\}$ is a fixed sets of numbers for which the conditions $p_j>0$ for all $j=\overline{0,q-1}$ and $p_0+p_1\dots + p_{q-1}=1$ hold. Then 
$$
\Delta^{P_q} _{i_1i_2...i_k...}:=\beta_{i_1}+ \sum^{\infty} _{k=2} {\left(\beta_{i_k}\prod^{k-1} _{r=1}{p_{i_r}}\right)}
$$
is \emph{the $P_q$-representation of $x\in[0, 1]$}.

Let us consider a class $S_q$ of modifications of Salem's functions of the form:
\begin{equation}
\label{ff1}
x=\Delta^{P_q} _{i_{1}i_{2}...i_{k}...}\stackrel{f_l}{\rightarrow} \Delta^{P_q} _{\theta_l(i_{1})\theta_l(i_{2})...\theta_l(i_{k})...}=f_l(x)=y,
\end{equation}
where the argument and values of the function are represented in terms of the $P_q$-representation, $\theta_l$ is a certain fixed permutation of digits $0, 1, \dots , q-1$, and 
values of the function are obtained from the $P_q$-representation of the argument by the change $\theta_l$ of digits. That is, there exist $q!$ different operators $\theta_l$ of digits. 

For the simplification, assume that  the function $f_1 (x)$ is the function  $y=x$ (i.e., this function is obtained by the following change of digits: $0$ by $0$,  $1$ by $1$, ..., $q-1$ by $q-1$) and the function $f_{q!} (x)$ is the function $y=\Delta^{P_q} _{[q-1-i_1][q-1-i_2]...[q-1-i_k]...}$ (this function is obtained by the following change of digits: $0$ by $q-1$,  $1$ by $q-2$, ..., $q-1$ by $0$), i.e.,  values of the function $f_l$  are obtained from the $P_q$-representation of the argument by the following change of digits: $0$ by $\theta_l(0)$,  $1$ by $\theta_l(1)$, ..., $q-1$ by $\theta_l(q-1)$, where $\theta_l(i)\ne\theta_(j)$ for any $i\ne j$. 

For the simplification of notations, for any $f\in S_q$ for  fixed $\theta_l$ and $l\in \{1, 2, 3, \dots , q!\}$ assume that
\begin{equation}
\label{ff2}
x=\Delta^{P_q} _{i_{1}i_{2}...i_{k}...}\stackrel{f}{\rightarrow} \Delta^{P_q} _{\theta(i_{1})\theta(i_{2})...\theta(i_{k})...}=\beta_{\theta(i_1)}+ \sum^{\infty} _{k=2} {\left(\beta_{\theta(i_k)}\prod^{k-1} _{r=1}{p_{\theta(i_r)}}\right)}=f(x)=y,
\end{equation}

Let us note that numbers of the form
$$
\Delta^{P_q} _{i_{1}i_{2}...i_{m-1}i_m000...}:=\Delta^{P_q} _{i_{1}i_{2}...i_{m-1}i_m(0)}=\Delta^{P_q} _{i_{1}i_{2}...i_{m-1}[i_m-1](q-1)}:=\Delta^{P_q} _{i_{1}i_{2}...i_{m-1}[i_m-1][q-1][q-1][q-1]...},
$$
where $i_m\ne 0$, are called \emph{$P_3$-rational}. The rest of numbers are called \emph{$P_3$-irrational}.

\begin{lemma}
For any function  $f$  from $S_{q}$ except for $f_1$ and $f_{q!}$, values of function $f$ for different representations of $P_3$-rational numbers from ~$[0;1]$  are different.
\end{lemma}
\begin{proof}
Consider an $P_3$-rational number
$$
x_1=\Delta^{P_q} _{i_{1}i_{2}...i_{m-1}i_m(0)}=\Delta^{P_q} _{i_{1}i_{2}...i_{m-1}[i_m-1](q-1)}=x_2, ~~~~~i_m\ne 0.
$$ 
It is easy to see that
$$
f\left(x_1\right)=\Delta^{P_q} _{\theta(i_{1})\theta(i_{2})...\theta(i_{m-1})\theta(i_m)(\theta(0))}= \Delta^{P_q} _{\theta(i_{1})\theta(i_{2})...\theta(i_{m-1})\theta(i_m-1)(\theta(q-1))}=f\left(x_2\right)
$$
is true whenever
$$
\left\{
\begin{aligned}
\theta(i_m)=\theta(i_m-1)+1\\
\left[
\begin{aligned}
\left\{
\begin{aligned}
\theta(0) & =0\\
\theta(q-1) & =q-1
\end{aligned}
\right.\\ 
\left\{
\begin{aligned}
\theta(0) & =q-1\\
\theta(q-1) & =0
\end{aligned}
\right.\\
\end{aligned}
\right.\\
\end{aligned}
\right.\\
$$
holds. Hence $f(x)=x$ or $f(x)=f_{q!}(x)$.
\end{proof}

\begin{remark}
\label{Remark3}
To reach that the function $f$ be well-defined on the set of $P_q$-rational numbers from $[0, 1]$, we will not consider the $P_q$-representations, which have period $(q-1)$ (without the number~$1$).
\end{remark}

\begin{lemma}
Any function  $f\in S_{q}$ except for $f_1$ and $f_{q!}$,  has the following properties:
\begin{enumerate}
\item $f$ maps the closed interval $[0,1]$ into $[0,1]$ without a certain enumerable subset, i.e.,
$$
f: [0,1]\stackrel{f}{\rightarrow} [0,1] \setminus \left\{y: y=\Delta^{P_q} _{j_1j_2...j_k(\theta(q-1))}\right\};
$$
\item according to the existence $i\in \{0, 1, \dots q-1\}$ for which the equality $\theta(i)=i$ holds, the set of invariant points is an empty set, an one-element set, or self-similar fractal;  
\item the function $f$ is not bijective on the domain.
\item $f$ is not a monotonic function on the domain.
\end{enumerate}
\end{lemma}
\begin{proof}
\emph{The first property} follows from Remark~\ref{Remark3}. \emph{The second property} follows from the definition of $f$ and fractal properties of the set having restrictions on using digits in the $P_q$-representations of elements (\cite{sets}), 
$$
E=\left\{x: x=\Delta^{P_q} _{i_ii_2...i_k...}, i_k\in \{r_1, r_2,\dots , r_t\}, 2<t<q, \theta(r_n)=r_n, n=\overline{1,t}\right\}
$$.
In addition, the Hausdorff dimension $\alpha_0(E)$ of the set $E$ satisfies the following equation (see \cite{sets}):
$$
\sum^{t} _{n=1}{\left(p_{r_n}\right)^{\alpha_0}}=1.
$$

Let us prove \emph{the third property}. Suppose $x_1=\Delta^{P_q} _{i_{1}i_{2}...i_{k}...}$ and $x_2=\Delta^{P_q} _{j_{1}j_{2}...j_{k}...}$ and $x_1\ne x_2$. Let us  find the following set
$$
\{x: f(x_1)=f(x_2), x_1 \ne x_2\}.
$$

If $y_0=f(x_1)=f(x_2)$ is $P_q$-irrational, then the following must be hold:
$$
y_0=\Delta^{P_q} _{\gamma_{1}\gamma_{2}...\gamma_{k}...}=\Delta^{P_q} _{\theta(i_{1})\theta(i_{2})...\theta(i_{k})...}=\Delta^{P_3} _{\theta(j_{1})\theta(j_{2})...\theta(j_{k})...}.
$$
Since the last equalities hold and $y_0$ is $P_q$-irrational, we have $x_1=x_2$ but this contradicts the condition $x_1\ne x_2$.

If $y_0$ is $P_q$-rational, then it is obvious that there exsist different numbers $x_1$ and $x_2$ such that
$$
f(x_1)=\Delta^{P_q} _{\gamma_{1}\gamma_{2}...\gamma_{m-1}\gamma_m000...}=\Delta^{P_q} _{\gamma_{1}\gamma_{2}...\gamma_{m-1}[\gamma_m-1][q-1][q-1][q-1]...}=f(x_2).
$$

\emph{The fourth property} follows from the following. Suppose $x_1<x_2$, as well as  $x_1=\Delta^{P_q} _{i_{1}i_{2}...i_{k}...}$ and $x_2=\Delta^{P_q} _{j_{1}j_{2}...j_{k}...}$, i.e., there exists $k_0$ such that $i_r=j_r$ for all $r=\overline{1,k_0}$ and $i_{k_0+1}<j_{k_0+1}$. Considering $f(x_1)$ and$f(x_2)$, one can note the following: the equality $\theta(i_r)=\theta(j_r)$ holds for all $r=\overline{1,k_0}$ but there are $i_{k_0+1}$ and $j_{k_0+1}$ from $\{0,1,\dots , q-1\}$ such that $\theta(i_{k_0+1})<\theta(j_{k_0+1})$ or $\theta(i_{k_0+1})>\theta(j_{k_0+1})$ as well. 
So, $f$ is not monotonic. 
\end{proof}

%%%%%%%%%%%%%%%%%%%%%%%%%%%%%%%%
\section{Differential properties}
%%%%%%%%%%%%%%%%%%%%%%%%%%%%%%%%

Let us begin with  definitions of auxiliary notions of a cylinder and the shift operator.
 
Let $c_1,c_2,\dots, c_m$ be a fixed ordered tuple of integers such that $c_r\in\{0,1,2\}$ for $r=\overline{1,m}$. 

\emph{A cylinder $\Lambda^{P_q} _{c_1c_2...c_m}$ of rank $m$ with base $c_1c_2\ldots c_m$} is the following set 
$$
\Lambda^{P_q} _{c_1c_2...c_m}\equiv\{x: x=\Delta^{P_q} _{c_1c_2...c_m i_{m+1}i_{m+2}\ldots i_{m+k}\ldots}\}.
$$
It is easy to see that  any cylinder $\Delta^{P_q} _{c_1c_2...c_m}$ is a closed interval of the form
$$
\left[\Delta^{P_q} _{c_1c_2...c_m(0)}, \Delta^{P_q} _{c_1c_2...c_m(q-1)}\right]=\left[\beta_{c_1}+\sum^{m} _{k=2}{\left(\beta_{c_r}\prod^{k-1} _{l=1}{p_{c_l}}\right)}, \beta_{c_1}+\sum^{m} _{k=2}{\left(\beta_{c_r}\prod^{k-1} _{l=1}{p_{c_l}}\right)}+\prod^{m} _{r=1}{p_{c_r}} \right].
$$
Whence,
$$
\left|\Lambda^{P_3} _{c_1c_2...c_m}\right|=\sup \Lambda^{P_3} _{c_1c_2...c_m}-\inf \Lambda^{P_3} _{c_1c_2...c_m}=\prod^{m} _{r=1}{p_{c_r}},
$$
where $|\cdot|$ is the Lebesgue measure of a set.

A  map $\sigma(x)$ of the following form  is called \emph{the shift operator} and is a piecewise linear function: 
$$
\sigma(x)=\sigma \left(\Delta^{P_q} _{i_1i_2...i_m...}\right)=\sigma^n \left(\Delta^{P_q} _{i_{2}i_{3}i_{4}...}\right)=\beta_{i_{2}}+\sum^{\infty} _{s=3}{\left(\beta_{i_s} \prod^{s-1} _{l=2}{p_{i_l}}\right)}.
$$
By analogy, 
$$
\sigma^n(x)=\sigma^n \left(\Delta^{P_q} _{i_1i_2...i_m...}\right)=\sigma^n \left(\Delta^{P_q} _{i_{n+1}i_{n+2}i_{n+3}...}\right)=\beta_{i_{n+1}}+\sum^{\infty} _{s=n+2}{\left(\beta_{i_s} \prod^{s-1} _{l=n+1}{p_{i_l}}\right)},
$$
$0\le \sigma^n(x)\le 1$.

\begin{lemma}
The function   $f\in S_q$ except for $f_1$ and $f_{q!}$, is continuous at $P_q$-irrational  points, and the $P_q$-rational  points are points of discontinuity of the function. Functions  $f_1$ and $f_{q!}$ are continuous on the closed interval $[0,1]$.
\end{lemma}
\begin{proof} Let use the standart technique considering two cases: the case when the argument is a $P_q$-rational number, as well as  the argument is a $P_q$-irrational number.

Let $x_0$ be a $P_3$-rational number, i.e.,
$$
x_0=\Delta^{P_q} _{i_1i_2...i_{m-1}i_m(0)}=\Delta^{P_q} _{i_1i_2...i_{m-1}[i_m-1](2)},~~~i_m \ne 0.
$$
Then 
$$
\lim_{x \to x_0 -0} {f(x)}=\Delta^{P_q} _{\theta(i_1)\theta(i_2)...\theta(i_{m-1})\theta(i_m-1)(\theta(q-1))},
$$
$$
\lim_{x \to x_0 +0} {f(x)}=\Delta^{P_q} _{\theta(i_1)\theta(i_2)...\theta(i_{m-1})\theta(i_m)(\theta(0))}.
$$
So, $x_0$ is a point of discontinuity. For the functions $f_1$ and $f_{q!}$, we get
$$
\lim_{x \to x_0 -0} {f(x)}=\lim_{x \to x_0 +0} {f(x)}.
$$

Suppose the following: $x_0, x \in \Lambda^{P_q} _{c_1c_2...c_m}$ are arbitrary $P_q$-irrational numbers; then  the condition $i_{r}=j_{r}$ holds  for all $j=\overline{1,m}$ but our numbers such that  $i_{m+1}\ne j_{m+1}$.  

Since $f$ is a bounded function, $0\le f(x) \le 1$, we obtain $g(x)-g(x_0)=$
$$
=\beta_{\theta(i_1)}+\sum^{m} _{k=2}{\left(\beta_{\theta(i_k)}\prod^{k-1} _{t=1}{p_{\theta{(i_t)}}}\right)}+\prod^{m} _{u=1}{p_{\theta(i_u)}}\left(\beta_{\theta(i_{m+1})}+\sum^{\infty} _{s=m+2}{\left(\beta_{\theta(i_s)}\prod^{s-1} _{l=m+1}{p_{\theta(i_l)}}\right)}\right)
$$
$$
-\beta_{\theta(i_1)}-\sum^{m} _{k=2}{\left(\beta_{\theta(i_k)}\prod^{k-1} _{t=1}{p_{\theta{(i_t)}}}\right)}-\prod^{m} _{u=1}{p_{\theta(i_u)}}\left(\beta_{\theta(j_{m+1})}+\sum^{\infty} _{s=m+2}{\left(\beta_{\theta(j_s)}\prod^{s-1} _{l=m+1}{p_{\theta(j_l)}}\right)}\right)
$$
$$
=\prod^{k_0} _{u=1}{p_{\theta(i_u)}}\left(\beta_{\theta(i_{m+1})}+\sum^{\infty} _{s=m+2}{\left(\beta_{\theta(i_s)}\prod^{s-1} _{l=m+1}{p_{\theta(i_l)}}\right)}-\beta_{\theta(j_{m+1})}-\sum^{\infty} _{s=m+2}{\left(\beta_{\theta(j_s)}\prod^{s-1} _{l=m+1}{p_{\theta(j_l)}}\right)}\right)
$$
$$
\le (1-0)\prod^{m} _{u=1}{p_{\theta(i_u)}}=\prod^{m} _{u=1}{p_{\theta(i_u)}}.
$$
Hence
$$
\lim_{m\to \infty}{|f(x)-f(x_0)|}=\lim_{m\to \infty}{\prod^{m} _{u=1}{p_{\theta(i_u)}}}\le\lim_{m\to \infty}{\left(\max\{p_0, p_1, p_2\}\right)^{m}}=0.
$$

So, $\lim_{x\to x_0}{f(x)}=f(x_0)$, i.e., any function $f\in S_q$ is continuous at any $P_q$-irrational point. 
\end{proof}

Let us consider a question about a fact that a function $f\in S_q$, but $ f(x)\ne x$, can be a singular function on a set  of the full Lebesgue measure.

\begin{conjecture}
A function $f\in S_q$ such that $f(x)\ne x$ is a singular function almost everywhere on~$[0, 1]$.
\end{conjecture}

Let us describe two techniques. 

For the first, let us consider a derivative on cylinders. Since a value of the increment $\mu_f (\cdot)$ of the fucnction $f$ on a set $\Lambda^{P_q} _{c_1c_2...c_m}$ can be calculated as following:
$$
\mu_f\left(\Lambda^{P_q} _{c_1c_2...c_m}\right)= f\left(\sup\Lambda^{P_q} _{c_1c_2...c_m}\right)- f\left(\inf \Lambda^{P_q} _{c_1c_2...c_m}\right)=\Delta^{P_q} _{\theta(c_1)\theta(c_2)...\theta(c_m)(\theta(q-1))}-\Delta^{P_q} _{\theta(c_1)\theta(c_2)...\theta(c_m)(\theta(0))}
$$
$$
=\left(\Delta^{P_q} _{(\theta(q-1))}-\Delta^{P_q} _{(\theta(0))}\right)\prod^{m} _{r=1}{p_{\theta(c_r)}},
$$
we get
$$
\lim_{m\to\infty}{\frac{\mu_f\left(\Lambda^{P_q} _{c_1c_2...c_m}\right)}{\left|\Lambda^{P_q} _{c_1c_2...c_m}\right|}}=\left(\Delta^{P_q} _{(\theta(q-1))}-\Delta^{P_q} _{(\theta(0))}\right)\lim_{m\to\infty}{\frac{\prod^{m} _{r=1}{p_{\theta(c_r)}}}{\prod^{m} _{r=1}{p_{c_r}}}}=\left(\Delta^{P_q} _{(\theta(q-1))}-\Delta^{P_q} _{(\theta(0))}\right)\lim_{m\to\infty}{\prod^{m} _{r=1}{\frac{p_{\theta(c_r)}}{p_{c_r}}}}.
$$

For the second, choose $x_0=\Delta^{P_q} _{i_1i_2...i_{n_0-1}ii_{n_0+1}i_{n_0+2}i_{n_0+3}...}$, where $i_{n_0}=i$ is a fixed digit; then let us  consider  a sequence $(x_n)$ such that  $x_n=\Delta^{P_q} _{i_1i_2...i_{n_0-1}ji_{n_0+1}i_{n_0+2}i_{n_0+3}...}$.

Suppose $n_0=1, 2, 3, \dots$; then
$$
x_n-x_0=\left(\prod^{n_0-1} _{r=1}{p_{i_r}}\right)\times
$$
$$
\times\left(\beta_j -\beta_i+p_j\left(\beta_{i_{n_0+1}}+\sum^{\infty} _{k=n_0+2}{\left(\beta_{i_k}\prod^{k-1} _{t=n_0+1}{p_{i_t}}\right)}\right)-p_i\left(\beta_{i_{n_0+1}}+\sum^{\infty} _{k=n_0+2}{\left(\beta_{i_k}\prod^{k-1} _{t=n_0+1}{p_{i_t}}\right)}\right)\right)
$$
$$
=\left(\prod^{n_0-1} _{r=1}{p_{i_r}}\right)\left(\beta_j -\beta_i+(p_j-p_i)\left(\beta_{i_{n_0+1}}+\sum^{\infty} _{k=n_0+2}{\left(\beta_{i_k}\prod^{k-1} _{t=n_0+1}{p_{i_t}}\right)}\right)\right)
$$
$$
=\left(\prod^{n_0-1} _{r=1}{p_{i_r}}\right)\left(\beta_j -\beta_i+(p_j-p_i)\sigma^{n_0}(x_n)\right),
$$
$$
f(x_n)-f(x_0)=\left(\prod^{n_0-1} _{r=1}{\theta(p_{i_r})}\right)\left(\beta_{\theta(j)} -\beta_{\theta(i)}+(p_{\theta(j)}-p_{\theta(i)})\sigma^{n_0}(f(x_n))\right),
$$
and
$$
\lim_{x_n-x_0\to 0}{\frac{f(x_n)-f(x_0)}{x_n-x_0}}=\lim_{n_0\to\infty}{\frac{f(x_n)-f(x_0)}{x_n-x_0}}
$$
$$
=\lim_{n_0\to\infty}{\frac{\beta_{\theta(j)} -\beta_{\theta(i)}+(p_{\theta(j)}-p_{\theta(i)})\sigma^{n_0}(f(x_n))}{\beta_j -\beta_i+(p_j-p_i)\sigma^{n_0}(x_n)}\prod^{n_0-1} _{r=1}{\frac{p_{\theta(c_r)}}{p_{c_r}}}}
$$

In our two cases, we have a certain number or a bounded sequence which multiplied on 
$$
\prod^{r} _{t=1}{\frac{p_{\theta(c_t)}}{p_{c_t}}},
$$
where $r\to\infty$. Let us evaluate this. 

Suppose $s\in\{0, 1, 2, \dots q-1\}$ and  $N_s(x, k)$ is the number of the digit $s$ in the $r$ first digits of the $P_q$-representation of $x$. Also, 
$$
\lim_{r\to\infty}{\frac{N_s(x,r)}{k}}=\nu_s(x)
$$
is \emph{the frequency of the digit $s$ in the $P_q$-representation of $x\in [0,1]$}. Using (for example \cite{{Salem1943}}, etc.) Salem's techniques (including the fact that the Salem function is continuous and strictly increasing for positive $p_0, p_1, \dots , p_{q-1}$),  as well as (see \cite{{HW1979}, {PVB2011}} and references therein) a statement  that for $q$-ary expansions the set of normal numbers (numbers such that $\nu_s =p_s=\frac 1 q $ holds for all digits~$s$) has the full Lebesgue measure, we obtain that on the set of the full  Lebesgue measure  the condition $\nu_s =p_s$ holds  for all digits $s$ for the $P_q$-representation.

So,
$$
\lim_{r\to\infty} \prod^{r} _{t=1}{\frac{p_{\theta(c_t)}}{p_{c_t}}}=\lim_{r\to\infty} \left(\frac{p^{N_0(x, r)} _{\theta(0)} p^{N_1(x,r)} _{\theta(1)} \cdots  p^{N_{q-1}(x,r)} _{\theta(q-1)}}{p^{N_0(x, r)} _0 p^{N_1(x, r)} _1 \cdots p^{N_{q-1}(x, r)} _{q-1}}\right)=\lim_{r\to\infty} \left(\frac{p^{\frac{N_0(x, r)}{r}} _{\theta(0)} p^{\frac{N_1(x,r)}{r}} _{\theta(1)} \cdots  p^{\frac{N_{q-1}(x,r)}{r}} _{\theta(q-1)}}{p^{\frac{N_0(x, r)}{r}} _0 p^{\frac{N_1(x, r)}{r}} _1 \cdots p^{\frac{N_{q-1}(x, r)}{r}} _{q-1}}\right)^r
$$
$$
=\lim_{r\to\infty} \left(\frac{p^{\nu_0} _{\theta(0)} p^{\nu_1} _{\theta(1)} \cdots  p^{\nu_{q-1}} _{\theta(q-1)}}{p^{\nu_0} _0 p^{\nu_1} _1 \cdots p^{\nu_{q-1}} _{q-1}}\right)^r=\lim_{r\to\infty} \left(\frac{p^{p_0} _{\theta(0)} p^{p_1} _{\theta(1)} \cdots  p^{p_{q-1}} _{\theta(q-1)}}{p^{p_0} _0 p^{p_1} _1 \cdots p^{p_{q-1}} _{q-1}}\right)^r=\lim_{r\to\infty}\left(\prod^{q-1} _{m=0}{\frac{p^{p_m} _{\theta(m)} \cdot p^{1-p_m} _{m}}{p_m}}\right)^r.
$$
Since $p_0+p_1+\dots +p_{q-1}=1$, $0<p_m<1$ for $m=\overline{0, q-1}$, and $p^\omega _m<1$ for $\omega >0$, as well as
$$
\lim_{p_m\to 1}{\frac{p^{p_m} _{\theta(m)} \cdot p^{1-p_m} _{m}}{p_m}}=p_{\theta(m)}<1
$$
and
$$
\lim_{p_m\to 0}{\frac{p^{p_m} _{\theta(m)} \cdot p^{1-p_m} _{m}}{p_m}}=1 
$$
hold, we have
$$
\prod^{q-1} _{m=0}{\frac{p^{p_m} _{\theta(m)} \cdot p^{1-p_m} _{m}}{p_m}}<1.
$$

Whence, 
$$
\lim_{r\to\infty} \prod^{r} _{t=1}{\frac{p_{\theta(c_t)}}{p_{c_t}}}=0
$$
and
$$
f^{'}(x_0)=\lim_{m\to\infty}{\frac{\mu_f\left(\Lambda^{P_q} _{c_1c_2...c_m}\right)}{\left|\Lambda^{P_q} _{c_1c_2...c_m}\right|}}=\lim_{x_n\to x_0}{\frac{f(x_n)-f(x_0)}{x_n-x_0}}=0.
$$

%%%%%%%%%%%%%%%%%%%%%%%%%
\section{Self-similarity, graph, and functional equations}
%%%%%%%%%%%%%%%%%%%%%%%%%

\begin{theorem} 
Let $P_q=(p_0, p_1, \dots , p_{q-1})$ be a fixed tuple of real numbers such that $p_t\in (0,1)$, where $t=\overline{0,q-1}$, $\sum_t {p_t}=1$, and $0=\beta_0<\beta_t=\sum^{t-1} _{l=0}{p_l}<1$ for all $t\ne 0$. Then the following system of functional equations
\begin{equation}
\label{eq: system}
f\left(\sigma^{n-1}(x)\right)=\beta_{\theta(i_{n})}+p_{\theta(i_{n})}f\left(\sigma^n(x)\right),
\end{equation}
where $x=\Delta^{P_q} _{i_1i_2...i_k...}$, $n=1,2, \dots$, $\sigma$ is the shift operator, and $\sigma_0(x)=x$, has the unique solution
$$
f(x)=\beta_{\theta({i_1})}+\sum^{\infty} _{k=2}{\left(\beta_{\theta(i_{k})}\prod^{k-1} _{r=1}{p_{\theta(i_{r})}}\right)}
$$
in the class of determined and bounded on $[0, 1]$ functions. 
\end{theorem}
\begin{proof} Since the function $f$ is a determined and bounded  function on $[0,1]$, using system~\eqref{eq: system},  we have
$$
f(x)=\beta_{\theta ({i_1})}+p_{\theta({i_1})}f(\sigma(x))=\beta_{\theta({i_1})}+p_{\theta({i_1})}(\beta_{\theta({i_2})}+p_{\theta({i_2})}f(\sigma^2(x)))=\dots
$$
$$
\dots =\beta_{\theta({i_1})}+\beta_{\theta({i_2})}p_{\theta({i_1})}+\beta_{\theta({i_3})}p_{\theta({i_1})}p_{\theta({i_2})}+\dots +\beta_{\theta{i_n}}\prod^{n-1} _{l=1}{p_{\theta({i_l})}}+\left(\prod^{n} _{r=1}{p_{\theta({i_r})}}\right)f(\sigma^n(x)).
$$
Since 
$$
\prod^{n} _{r=1}{p_{\theta(i_r)}}\le \left( \max\{p_0, p_1, \dots , p_{q-1}\}\right)^n\to 0, ~~~ n\to \infty,
$$
and
$$
\lim_{n\to\infty}{f(\sigma^n(x))\prod^{n} _{r=1}{p_{\theta({i_r})}}}=0,
$$
we have
$$
f(x)=\beta_{\theta(i_1)}+\sum^{\infty} _{k=2}{\left(\beta_{\theta({i_k})}\prod^{k-1} _{r=1}{p_{\theta({i_r})}}\right)}.
$$
\end{proof}

\begin{theorem}
Suppose
$$
\psi_t: \left\{
\begin{array}{rcl}
x^{'}&=&p_tx+\beta_t\\
y^{'} & = &p_{\theta(t)}y+\beta_{\theta(t)}\\
\end{array}
\right.
$$
are affine transformations for $t=0, 1, \dots , q-1$ and $p_0, p_1, \dots , p_{q-1} \in (0, 1)$. Then the graph $\Gamma$ of  $f$ is a self-affine set of $\mathbb R^2$, as well as
$$
\Gamma= \bigcup^{q-1} _{t=0}{\psi_t(\Gamma)}.
$$
\end{theorem}
\begin{proof} By analogy with~\cite{Symon2023}, we obtain the following. If $M(x_0, y_0)\in T\subset \Gamma$, then  $x_0=p_tx+\beta_t$ and $y_0=p_{\theta(t)} y+\beta_{\theta(t)}$ for some $t\in \{0, 1, \dots , q-1\}$. Using the system of functional equations~\eqref{eq: system}, we get
$$
f(x_0)=\beta_{\theta(t)}+p_{\theta(t)}f(x)=y_0
$$ 
and $M\in \Gamma$.

Choose  $M(x_0, f(x_0))\in \Gamma$. Then $x_0=\beta_t+p_t \sigma(x_0)$, $f(x_0)=\beta_{\theta(t)}+p_{\theta(t)}f(\sigma(x_0))$, and $(\sigma(x_0), f(\sigma(x_0)))\in \Gamma$.

So, $\psi_t(\sigma(x_0), f(\sigma(x_0)))=(x_0, f(x_0))\in T$.
\end{proof}

Let us consider fractal properties of the graph $\Gamma_f$ of the function $f$.

\begin{theorem}
Suppose $a=\min \{p_0, p_{\theta(0)}, \dots , p_{q-1}, p_{\theta(q-1)}\}$ and $b=\max \{p_0, p_{\theta(0)}, \dots , p_{q-1}, p_{\theta(q-1)}\}$. 

Then for a value  $\alpha_0$ of the Hausdorff dimension of the graph $\Gamma_f$ of the function $f$, the inequality $\alpha_{2}\le \alpha_0\le\alpha_{1}$ holds, where 
$$
\alpha_1=\inf\left\{\alpha: q\left(a\right)^{\alpha}<1\right\}
$$
and
$$
\alpha_2=\inf\left\{\alpha: q\left(b\right)^{\alpha}<1\right\}.
$$
\end{theorem}
\begin{proof} Using the definition of $f$, one can see that the graph of this function belongs to $q$ rectangles (with sides $p_{i_1}$ and $p_{\theta(i_1)}$) from  $q^2$
first-rank rectangles:
$$
R_{[i_1, \theta(i_1)]}=\left[\beta_{i_1}, \beta_{i_1+1}\right]\times\left[\beta_{\theta(i_1)}, \beta_{\theta(i_1)+1}\right],
$$
where $0\le \beta_{i_1+1}, \beta_{\theta(i_1)+1}\le \beta_q=1$ and $i_1\in\{0, 1, \dots, q-1\}$.
The sum of all areas of these rectangles is equal to 
$$
L_1:=\sum^{q-1} _{m=0}{p_mp_{\theta(m)}}.
$$

The graph of  the function $f$ belongs to $q^2$ rectangles (with sides $p_{i_1}p_{i_2}$ and $p_{\theta(i_1)}p_{\theta(i_2)}$) from $q^4$ rectangles of rank $2$:
$$
R_{[i_1, \theta(i_1)][i_2,\theta(i_2)]}=\left[\beta_{i_1}+\beta_{i_2}p_{i_1}, \beta_{i_1}+\beta_{i_2+1}p_{i_1}\right] \times\left[\beta_{\theta(i_1)}+\beta_{\theta(i_2)}p_{\theta(i_1)}, \beta_{\theta(i_1)}+\beta_{\theta(i_2)+1}p_{\theta(i_1)}\right].
$$
In addition, one can note the following:
\begin{itemize}
\item the part of the graph, which is in the rectangle $R_{[0, \theta(0)]}$, belongs to $q$ rectangles  
$$
R_{[0,\theta(0)][0, \theta(0)]}, R_{[0, \theta(0)][(1, \theta(1)]}, \dots , R_{[0, \theta(0)][q-1, \theta(q-1)]};
$$
\item  the part of the graph, which is in the rectangle $R_{[1, \theta(1)]}$, belongs to $q$ rectangles  
$$
R_{[1,\theta(1)][0, \theta(0)]}, R_{[1, \theta(1)][(1, \theta(1)]}, \dots , R_{[1, \theta(1)][q-1, \theta(q-1)]};
$$
\item $\dots \dots \dots \dots \dots \dots \dots \dots \dots \dots \dots \dots \dots \dots \dots \dots \dots \dots \dots \dots \dots$
\item  the part of the graph, which is in the rectangle $R_{[q-1, \theta(q-1)]}$, belongs to $q$ rectangles  
$$
R_{[q-1,\theta(q-1)][0, \theta(0)]}, R_{[q-1, \theta(q-1)][(1, \theta(1)]}, \dots , R_{[q-1, \theta(q-1)][q-1, \theta(q-1)]}.
$$
\end{itemize}
The sum of all areas of these rectangles is equal to 
$$
L_1\sum^{q-1} _{m=0}{p_mp_{\theta(m)}}=L^2 _1.
$$

In the $r$th step, the graph of  the function $f$ belongs to $q^r$ rectangles (with sides $\prod^{r} _{t=1}{p_{i_t}}$ and $\prod^{r} _{t=1}{p_{\theta(i_t)}}$) from $q^2r$ rectangles of rank $r$. The sum of all areas of these rectangles is equal to $L^r _1$.

Whence,
$$
\widehat{H}_{\alpha} (\Gamma_f)=\lim_{\overline{r \to \infty}}{\sum_{j=\overline{1, r},~~ p_{i_j}\in  P}\left(\sqrt{\left(\prod^{r} _{j=1}{p_{i_j}}\right)^2+\left(\prod^{r} _{j=1}{p_{\theta(i_j)}}\right)^2}\right)^{\alpha^K (\Gamma_f)}},
$$
where $\alpha^K (E)$ is the fractal cell entropy dimension of the set $E$ (see \cite{Serbenyuk-2016} and references therein).

Suppose $a=\min \{p_0, p_{\theta(0)}, \dots , p_{q-1}, p_{\theta(q-1)}\}$ and $b=\max \{p_0, p_{\theta(0)}, \dots , p_{q-1}, p_{\theta(q-1)}\}$. The value of $\widehat{H}_{\alpha} (\Gamma_f)$ is between values of $\widehat{H}_{\alpha} (\Gamma_f)$ for the case of \emph{squares} of rank $r$ with sides $a^r$ and $b^r$ correspondently. Hence using the number of such squares and their diameters, let us evaluate the following values. Really, 
$$
\widehat{H}_{\alpha, a} (\Gamma_f)=\lim_{\overline{r \to \infty}}{q^r\left(\sqrt{(a^r)^2+(a^r)^2}\right)^{\alpha}}=\lim_{\overline{r \to \infty}}{q^r\left(\sqrt{2(a^r)^2}\right)^{\alpha}}=\lim_{\overline{r \to \infty}}\left(2^{\frac{\alpha}{2}}\left(qa^{\alpha}\right)^{r}\right) 
$$
and
$$
\widehat{H}_{\alpha, b} (\Gamma_f)=\lim_{\overline{r \to \infty}}{q^r\left(\sqrt{(b^r)^2+(b^r)^2}\right)^{\alpha}}=\lim_{\overline{r \to \infty}}{q^r\left(\sqrt{2(b^r)^2}\right)^{\alpha}}=\lim_{\overline{r \to \infty}}\left(2^{\frac{\alpha}{2}}\left(qb^{\alpha}\right)^{r}\right).
$$
Since the graph of our function has self-similar properties,  $\left(qa^{\alpha}\right)^{r}\to 0$ and$\left(qb^{\alpha}\right)^{r}\to 0$ for $r\to \infty$ and for large $\alpha>1$, we obtain $\alpha_2 \le \alpha_0(\Gamma_f)\le \alpha_1$, where $\alpha_0$ is the Hausdorff dimension, as well as 
$$
\alpha_1=\inf_{\alpha}\left\{\alpha: q\left(a\right)^{\alpha}<1\right\}
$$
and
$$
\alpha_2=\inf_{\alpha}\left\{\alpha: q\left(b\right)^{\alpha}<1\right\}.
$$

For example, for the case when $p_0=p_1=\dots =p_{q-1}=\frac 1 q$, we have $\alpha_1=\alpha_2=1$. Hence $\alpha_0(\Gamma_f)=1$.
\end{proof}

%%%%%%%%%%%%%%%%%%%%%%%%%%%%%%%%
\section{Integral properties}
%%%%%%%%%%%%%%%%%%%%%%%%%%%%%%%%

\begin{theorem}
For the Lebesgue integral, the following equality holds:
$$
\int^1 _0 {f(x)dx}=\frac{\sum^{q-1}_{t=0}{\beta_{\theta(t)}p_t}}{1-\sum^{q-1} _{t=0}{p_{\theta(t)}p_t}}.
$$
\end{theorem}
\begin{proof}
 Let us begin with some equalities which are useful for the future calculations:
$$
x=\beta_{i_1}+p_{i_1}\sigma(x)
$$
and
$$
dx=p_{i_1}d(\sigma(x)),
$$
as well as
$$
d(\sigma^{n-1} (x))=p_{i_n}d(\sigma^{n} (x))
$$
for all $n=1, 2, 3, \dots$

Let us calculate the Lebesgue integral
$$
I:=\int^1 _0{f(x)dx}=\sum^{q-1} _{t=0}{\int^{\beta_{t+1}} _{\beta_t}{f(x)dx}}=\sum^{q-1} _{t=0}{\int^{\beta_{t+1}} _{\beta_t}{(\beta_{\theta(t)}+p_{\theta(t)}f(\sigma(x))})dx}
$$
$$
=\beta_{\theta(0)}p_0+\beta_{\theta(1)}p_1+\dots + \beta_{\theta(q-1)}p_{q-1}+\sum^{q-1} _{t=0}{\int^{\beta_{t+1}} _{\beta_t}{p_{\theta(t)}f(\sigma(x))dx}}
$$
$$
=\sum^{q-1}_{t=0}{\beta_{\theta(t)}p_t}+\sum^{q-1} _{t=0}{p_{\theta(t)}\int^{\beta_{t+1}} _{\beta_t}{p_tf(\sigma(x))d(\sigma(x))}}
$$
$$
=\sum^{q-1}_{t=0}{\beta_{\theta(t)}p_t}+\sum^{q-1} _{t=0}{p_{\theta(t)}p_t\int^{\beta_{t+1}} _{\beta_t}{f(\sigma(x))d(\sigma(x))}}.
$$

Using self-affine properties of $f$, we get
$$
I=\sum^{q-1}_{t=0}{\beta_{\theta(t)}p_t}+I\sum^{q-1} _{t=0}{p_{\theta(t)}p_t}.
$$

So,
$$
I=\frac{\sum^{q-1}_{t=0}{\beta_{\theta(t)}p_t}}{1-\sum^{q-1} _{t=0}{p_{\theta(t)}p_t}}.
$$
\end{proof}

\section*{Statements and Declarations}
\begin{center}
{\bf{Competing Interests}}

\emph{The author states that there is no conflict of interest}
\end{center}

\end{document}